\documentclass[final,hidelinks,onefignum,onetabnum]{siamart251216}

\usepackage{cite}
\usepackage{subcaption}
\usepackage{booktabs}
\usepackage{mathtools}
\usepackage{amssymb}
\usepackage{xcolor}
\usepackage{needspace}
\usepackage{amsfonts}
\usepackage{graphicx}
\usepackage{epstopdf}
\usepackage{algorithmic}
\ifpdf
  \DeclareGraphicsExtensions{.eps,.pdf,.png,.jpg}
\else
  \DeclareGraphicsExtensions{.eps}
\fi

\newsiamremark{remark}{Remark}
\newsiamremark{hypothesis}{Hypothesis}
\crefname{hypothesis}{Hypothesis}{Hypotheses}
\newsiamthm{claim}{Claim}
\newsiamremark{fact}{Fact}
\crefname{fact}{Fact}{Facts}

\headers{Toeplitz Infinite GMRES}{W. Gao and F. Jiang}

\title{Toeplitz Infinite GMRES for Parameterized Linear Systems}

\author{Weiguo Gao\thanks{School of Mathematical Sciences, Fudan University, Shanghai 200433, China \& Shanghai Key Laboratory of Contemporary Applied Mathematics, Shanghai, China, 200433 (\email{wggao@fudan.edu.cn}).}
\and Feiyang Jiang\thanks{School of Mathematical Sciences, Fudan University, Shanghai 200433, China (\email{fyjiang26@m.fudan.edu.cn}).}}

\usepackage{amsopn}
\DeclareMathOperator{\diag}{diag}

\newcommand{\cK}{\mathcal{K}}

\newcommand{\Ibar}[1]{\begin{bmatrix}I_{#1}\\0_{1\times #1}\end{bmatrix}}
\newcommand{\Jbar}[1]{\begin{bmatrix}0_{1\times #1}\\I_{#1}\end{bmatrix}}
\DeclareMathOperator{\Range}{Range}

\allowdisplaybreaks

\ifpdf
\hypersetup{
  pdftitle={Toeplitz Infinite GMRES for Parameterized Linear Systems},
  pdfauthor={Weiguo Gao and Feiyang Jiang}
}
\fi

\begin{document}
\maketitle

\begin{abstract}
We develop Toeplitz infinite GMRES for solving large sparse analytic parameterized 
systems $A(s)x(s)=z$ at many parameter values. Our method exploits a block upper triangular 
Toeplitz structure in the companion Krylov sequence to construct the Arnoldi process and recover 
solution approximations without storing full Arnoldi vectors. We derive incremental recurrences 
requiring $\mathcal{O}(np^2+p^3)$ arithmetic and $\mathcal{O}(np+p^2)$ storage for $p$ Arnoldi steps, 
excluding factorization setup and assuming linear-cost coefficient actions and triangular solves. 
A dynamic generator-refreshing strategy addresses cancellation in this basic recurrence. Moreover, 
we show that the Arnoldi basis can be simply represented using one block row and a nilpotent matrix. 
For $A(s)$ admitting a split form, we further use this structure to develop an implicitly rebased method. 
A residual lower bound for square-summable Taylor coefficients explains the effect of scaling outside the unit disk. 
Theoretical analysis and numerical experiments illustrate the computational efficiency of the proposed methods.
\end{abstract}



\section{Introduction}
\label{sec:introduction}

We consider the repeated solution of
\[
  A(s)x(s)=z, \qquad s\in\mathbb{C},
\]
where $A(s)\in\mathbb{C}^{n\times n}$ is large, sparse, and analytic, and
$z\in\mathbb{C}^n$ is fixed.  Such systems arise in parameter-dependent partial
differential equations, transfer-function evaluation, and contour-integral
methods for nonlinear eigenvalue problems.  When many parameter values are
required, factoring each matrix $A(s)$ separately is often too expensive.

GMRES\cite{Saad1986} methods for shifted linear systems exploit Krylov shift
invariance to reuse a common search space \cite{Frommer1998}.
Infinite GMRES (infGMRES) \cite{Jarlebring2022} applies this
principle to analytic parameterized systems through a Taylor
companion linearization.  It constructs one Arnoldi factorization for all parameter
values and solves only a small parameter-dependent least-squares problem for
each $s$. Extensions include inexact preconditioning \cite{Correnty2024} and
a two-level orthogonal Arnoldi (TOAR) implementation for contour-integral
nonlinear eigensolvers \cite{Liu2025}.  Related compact representations
appear in infinite and tensor infinite Arnoldi methods \cite{Jarlebring2014,Jarlebring2017,Jarlebring2012,Mele2018,VanBeeumen2016}
and in TOAR methods \cite{Zhang2013,Lu2016} for polynomial
linearizations \cite{Kressner2014}, as well as in the compact rational Krylov (CORK)
framework \cite{VanBeeumen2015}.

Constructing this Arnoldi basis is still expensive.  Suppose that
coefficient-matrix products and triangular solves using one stored LU
factorization require $\mathcal{O}(n)$ arithmetic. For $p$ steps,
direct infGMRES then has
$\mathcal{O}(np^3)$ arithmetic and $\mathcal{O}(np^2)$ storage complexity. With TOAR techniques, TOAR infGMRES reduces storage 
complexity to $\mathcal{O}(np+p^3)$, but still needs $\mathcal{O}(np^3+p^4)$ arithmetic complexity due to the construction of full Arnoldi vectors. 

Our approach exploits the Toeplitz structure in the 
companion Krylov sequence and updates only the QR data needed for solution recovery. Fast QR methods for scalar and block Toeplitz matrices \cite{Bojanczyk1986,Chandrasekaran1998} motivate our approach. 
We first derive an incremental recurrence for the needed QR data, with only $\mathcal O(np^2+p^3)$ arithmetic and $\mathcal O(np+p^2)$ storage complexity. However, this approach 
will suffer from numerical cancellation. So we address this cancellation through dynamic generator-refreshing, even though it will add extra computational cost. The dynamic method further motivates us to use the first block row of the orthogonal basis of the Toeplitz matrix as the basis to construct the Arnoldi vectors. Based on this idea, we give a structured representation of the Arnoldi basis.
Under an additional split form assumption of $A(s)$, we use this structured representation to give the implicitly rebased method, which is both fast and stable in numerical experiments. 

The paper is organized as follows.
Section~\ref{sec:infGMRES} reviews infGMRES.
Section~\ref{sec:algorithms} uses the block Toeplitz structure to 
derive the incremental and dynamic algorithms.
Section~\ref{sec:implicit-rebasing} further explores the structure of the Arnoldi basis and develops the implicitly rebased method.
Section~\ref{sec:residual-floor} establishes the residual lower bound of infGMRES and
explains the effect of scaling.
Section~\ref{sec:experiments} applies our algorithms to several numerical experiments to 
illustrate their performance.
Conclusions are given in Section \ref{sec:conclusions}.

\section{Review of infGMRES}
\label{sec:infGMRES}
We first introduce the infGMRES framework \cite{Correnty2024}.
Throughout, assume that $z\ne0$, $A_0\coloneqq A(0)$ is nonsingular, and $A(s)$ is
analytic near the origin.  Its Taylor expansion is
\begin{equation}
  A(s)=\sum_{j=0}^{\infty}A_js^j,
  \qquad A_j\coloneqq \frac{A^{(j)}(0)}{j!},
  \label{eq:taylor}
\end{equation}
and its degree-$q$ truncation is $A_q(s)=\sum_{j=0}^q A_js^j$.  For any $s$
such that $A_q(s)$ is nonsingular, the solution $x_q(s)$ of $A_q(s)x_q(s)=z$ is equivalent to
the first $n$ elements of the solution of
\[
  (\mathcal L_0-s\mathcal L_1)v(s)=\operatorname{vec}_0(z),
  \qquad \operatorname{vec}_0(z)\coloneqq
  \begin{bmatrix}z^T&0&\cdots&0\end{bmatrix}^{T}
  \in\mathbb{C}^{(q+1)n},
\]
with
\[
\mathcal L_0\coloneqq
\begin{bmatrix}
 A_0&A_1&A_2&\cdots&A_q\\
&I&&&\\
&&I&&\\
&&&\ddots&\\
&&&&I
\end{bmatrix},
\quad
\mathcal L_1\coloneqq
\begin{bmatrix}
0&0&\cdots&0&0\\
I&0&&&\\
&I&\ddots&&\\
&&\ddots&0&\\
&&&I&0
\end{bmatrix}.
\]
Right preconditioning with $\mathcal L_0^{-1}$ gives
\begin{equation}
  (I-s\mathcal L_1\mathcal L_0^{-1})\widehat v(s)=\operatorname{vec}_0(z), \qquad \widehat v=\mathcal L_0v.
  \label{eq:shifted-system}
\end{equation}

Then we apply GMRES to solve \eqref{eq:shifted-system}. Let $p\leq q$ be the maximum number of iterations.  At
iteration $m\leq p$, the Arnoldi process applied to $\mathcal L_1\mathcal L_0^{-1}$ with $\operatorname{vec}_0(z)/\|z\|_2$ gives
\begin{equation}
  \mathcal L_1\mathcal L_0^{-1}\mathcal U_m=\mathcal U_{m+1}\underline H_m,
  \qquad \mathcal U_m^*\mathcal U_m=I_m.
  \label{eq:arnoldi}
\end{equation}
Next, for each $s$, we solve the reduced least-squares problem
\begin{equation}
  y_s=\operatorname*{argmin}_{y\in\mathbb{C}^m}
  \left\|\left(\Ibar{m}-s\underline H_m\right)y-\|z\|_2 e_1\right\|_2.
  \label{eq:small-ls}
\end{equation}
Then $\widetilde x(s) =[I_n, 0, \ldots, 0] \mathcal L_0^{-1}\mathcal U_my_s$ approximates $x_q(s)$. 

The name \emph{infinite GMRES} comes from extending this construction to an
infinite-dimensional companion linearization
\cite{Jarlebring2022}.  The starting vector $\operatorname{vec}_0(z)/\|z\|_2$
has only its first block nonzero.  Multiplication by $\mathcal L_1\mathcal L_0^{-1}$ increases the number
of potentially nonzero blocks by at most one, and Arnoldi orthogonalization
preserves the remaining zero blocks.  Thus, in exact arithmetic, the first
$p$ Arnoldi steps are identical for every $q$ greater than or equal to $p$, up to appending zero
blocks to the basis vectors.  The reduced least-squares problems and the
recovered approximations $\widetilde x(s)$ are consequently independent of
$q$ as well.  We may therefore formally let $q\to\infty$ and interpret these
iterates as GMRES applied to the infinite companion linearization of
the full Taylor series.  Each iteration still requires only finitely many
active blocks and Taylor coefficients, so the infinite matrix is never
formed.  Omitting
the zero tails introduces no truncation error into these iterates.

\section{Toeplitz infinite GMRES}
\label{sec:algorithms}

\subsection{Block Toeplitz motivation}
\label{sec:toeplitz-view}

The preconditioned companion matrix is
\[
\mathcal L_1\mathcal L_0^{-1}=
\begin{bmatrix}
0&0&0&\cdots&0\\
 A_0^{-1}&-A_0^{-1}A_1&-A_0^{-1}A_2&\cdots&-A_0^{-1}A_q\\
&I&0&&\\
&&\ddots&\ddots&\\
&&&I&0
\end{bmatrix}.
\]

Now we consider the structure of $\cK_m(\mathcal L_1\mathcal L_0^{-1},\operatorname{vec}_0(z))$. 
Define
\begin{equation}
 a_1\coloneq A_0^{-1}z,
 \qquad
 a_m\coloneq -A_0^{-1}\sum_{j=1}^{m-1}A_ja_{m-j},
 \quad \forall m\geq 2.
 \label{eq:a-recurrence}
\end{equation}
Direct multiplication gives
\[
\cK_m(\mathcal L_1\mathcal L_0^{-1},\operatorname{vec}_0(z))=\Range
\begin{bmatrix}
z&0\\
0&V_{m-1}\\
0&0
\end{bmatrix},
\]
where 
\[
V_m\coloneqq
\begin{bmatrix}
a_1&a_2&a_3&\cdots&a_m\\
&a_1&a_2&\cdots&a_{m-1}\\
&&\ddots&\ddots&\vdots\\
&&&a_1&a_2\\
&&&&a_1
\end{bmatrix}\in\mathbb{C}^{mn\times m},\quad \forall m.
\]

Because $a_1\ne0$, so $V_m$ has full column
rank.  We use its unique thin QR factorization $V_m=Q_mR_m$, with a positive
real diagonal in $R_m$.  These factorizations are nested because the first
$m-1$ columns of $V_m$ equal $[V_{m-1}^T,0]^T$.

This structure already reduces infGMRES to QR factorization of Toeplitz matrix $V_m$. 
Using the coordinates under $V_m$, we can prove that the reduced least-squares problem 
and approximate solution in infGMRES are
\begin{equation}
 \min_{y\in\mathbb{C}^m}
 \left\|\begin{bmatrix}\|z\|_2&0\\0&R_m\end{bmatrix}
 \left(\Ibar{m}-s\Jbar{m}\right)y-\|z\|_2 e_1\right\|_2,
 \quad
 \widetilde x(s)=[a_1,\ldots,a_m]y.
 \label{eq:direct-toeplitz-coordinates}
\end{equation}

To see this, we represent a vector in $\cK_m(\mathcal L_1\mathcal L_0^{-1},\operatorname{vec}_0(z))$ by
\[
 v_m(y)=
 \begin{bmatrix}z&0\\0&V_{m-1}\\0&0\end{bmatrix}y.
\]
Its lifted residual satisfies
\begin{align*}
 &(I-s\mathcal L_1\mathcal L_0^{-1})v_m(y)-\operatorname{vec}_0(z)\\
 &=\begin{bmatrix}z&0\\0&V_{m}\\0&0\end{bmatrix}\left(\Ibar{m}-s\Jbar{m}\right)y-\operatorname{vec}_0(z)\\
 &=\begin{bmatrix}z/\|z\|_2&0\\0&Q_{m}\\0&0\end{bmatrix}
 \left(\begin{bmatrix}\|z\|_2&0\\0&R_m\end{bmatrix}
 \left(\Ibar{m}-s\Jbar{m}\right)y-\|z\|_2 e_1\right).
\end{align*}
The columns of $\diag(z/\|z\|_2,Q_m)$ are orthonormal, so this factorization
preserves the residual norm and leads to the minimization in \eqref{eq:direct-toeplitz-coordinates}.
Moreover, the first block row of $\mathcal L_0^{-1}v_m(y)$ is
$[a_1,\ldots,a_m]y$, which gives the approximate solution. 

A direct QR of $V_m$ can be implemented in $\mathcal{O}(nm^2+m^3)$ arithmetic operations by modifying
the previous Toeplitz QR methods \cite{Bojanczyk1986,Chandrasekaran1998}. However, these methods are undesirable when $\{a_k\}_{k=1}^m$ becomes nearly aligned or increases exponentially,  
which can be caused by the recursive formula \eqref{eq:a-recurrence}.  In these cases, the QR factorization of $V_m$ will lose accuracy. Meanwhile, the ill-conditioned $R_m$ will also make minimizing \eqref{eq:direct-toeplitz-coordinates} unreliable.
We therefore turn to the original Arnoldi process of infGMRES and develop an incremental recurrence that updates only the QR data needed for solution recovery.
The following lemma shows the needed QR data.

\begin{lemma}[Arnoldi process from Toeplitz QR]
\label{lem:arnoldi-from-toeplitz}
For $m\ge2$, let $V_m=Q_mR_m$ be a thin QR factorization,
then the Hessenberg matrix in \eqref{eq:arnoldi} is
\begin{equation}
 \underline H_m=
 \begin{bmatrix}
 0&0_{1\times(m-1)}\\[2mm]
 \dfrac{\|a_1\|_2}{\|z\|_2}e_1&R_m(1\!:\!m, 2\!:\!m)R_{m-1}^{-1}
 \end{bmatrix}.
 \label{eq:H-from-S}
\end{equation}
Moreover, if $y_s$ minimizes \eqref{eq:small-ls}, then the approximate solution is
\begin{equation}
 \widetilde x(s)=
 \left[\frac{a_1}{\|z\|_2},
 [a_2,\ldots,a_m]R_{m-1}^{-1}\right]y_s.
 \label{eq:solution-from-S}
\end{equation}
\end{lemma}

\begin{proof}
The Krylov structure and the nested QR factorizations give
\[
\mathcal U_m=
\begin{bmatrix}
\dfrac{z}{\|z\|_2}&0\\
0&Q_{m-1}\\
0&0
\end{bmatrix}=
\begin{bmatrix}
z&0\\
0&V_{m-1}\\
0&0
\end{bmatrix}
\begin{bmatrix}
  \dfrac{1}{\|z\|_2}&0\\
  0&R_{m-1}^{-1}
\end{bmatrix}
\]
By direct multiplication, we obtain
\begin{align*}
\mathcal L_1\mathcal L_0^{-1}\mathcal U_m
&=\begin{bmatrix}
  0_{n\times m}\\
  V_m
\end{bmatrix}
\begin{bmatrix}
  \dfrac{1}{\|z\|_2} & \\
  & R_{m-1}^{-1}
\end{bmatrix}\\
&=\begin{bmatrix}
  0_{n\times m}\\
  Q_m
\end{bmatrix}
\begin{bmatrix}
  \dfrac{\|a_1\|_2}{\|z\|_2}e_1 & R_m(:,2\!:\!m)R_{m-1}^{-1}
\end{bmatrix}\\
&=\mathcal U_{m+1}
\begin{bmatrix}
0&0\\
\dfrac{\|a_1\|_2}{\|z\|_2}e_1&R_m(1\!:\!m,2\!:\!m)R_{m-1}^{-1}
\end{bmatrix}.
\end{align*}
This proves
\eqref{eq:H-from-S}. Finally,
\begin{align*}
[I_n,0,\ldots,0]\mathcal L_0^{-1}\mathcal U_m
&=\left[\frac{a_1}{\|z\|_2},
-A_0^{-1}[A_1,\ldots,A_{m-1}]Q_{m-1}\right]\\
&=\left[\frac{a_1}{\|z\|_2},
-A_0^{-1}[A_1,\ldots,A_{m-1}]V_{m-1}R_{m-1}^{-1}\right]\\
&=\left[\frac{a_1}{\|z\|_2},
[a_2,\ldots,a_m]R_{m-1}^{-1}\right].
\end{align*}
Multiplying by $y_s$ proves \eqref{eq:solution-from-S}.
\end{proof}

Define
\[
 S_m\coloneq R_m(1\!:\!m,2\!:\!m)R_{m-1}^{-1}
 \in\mathbb{C}^{m\times(m-1)}.
\]
From Lemma~\ref{lem:arnoldi-from-toeplitz}, solution recovery
requires only $S_m$ and the $n$-dimensional recovery vectors
$[u_2,\ldots,u_m]\coloneqq [a_2,\ldots,a_m]R_{m-1}^{-1}$.  Together with
$a_1/\|z\|_2$, these vectors form the recovery basis in
\eqref{eq:solution-from-S}. Computing full Arnoldi vectors is unnecessary.

Let $q_{i,m} \coloneq Q_m((i-1)n+1\!:\!in, m)$, $i = 1,\ldots,m$, 
which is the $i$-th block row, $m$-th column of $Q_m$. 
In the proof of the above lemma, we have seen that
\[
[a_2,\ldots,a_{m}]R_{m-1}^{-1}
 =-A_0^{-1}[A_1,\ldots,A_{m-1}]Q_{m-1}.
\]
By taking the last column, we can derive that
\begin{equation}
u_m = -A_0^{-1}\sum_{i=1}^{m-1}A_iq_{i,m-1},
\label{eq:recovery-block-action}
\end{equation}
which can be an equivalent definition of $u_m$.

\subsection{Incremental recurrences}
\label{sec:recurrences}
We derive recurrences for $q_{1,m}$, $S_m$, and $u_m$. Let
\[g_1=\begin{bmatrix}a_1/\|a_1\|\\0_{m-1}\end{bmatrix}, \qquad
P_1=I-\frac{a_1a_1^*}{a_1^*a_1},
\]
where $P_1$ is a projection operator.
Assume that all quantities through step $m-1$ are available.

\paragraph{Updating $q_{1,m}$ and $S_m$}
By the block Toeplitz structure of $V_m$, the definition
$[u_2,\ldots,u_m]=[a_2,\ldots,a_m]R_{m-1}^{-1}$, and
$V_{m-1}=Q_{m-1}R_{m-1}$, we have
\begin{align}
 V_m
 &=\begin{bmatrix}
 a_1&[a_2,\ldots,a_m]\\
 0&V_{m-1}
 \end{bmatrix}\nonumber\\
 &=\begin{bmatrix}
 I_n&0\\0&Q_{m-1}
 \end{bmatrix}
 \begin{bmatrix}
 a_1&[u_2,\ldots,u_m]\\0&I_{m-1}
 \end{bmatrix}
 \begin{bmatrix}
 1&0\\0&R_{m-1}
 \end{bmatrix}.
 \label{eq:Vm-three-factor}
\end{align}
The leftmost factor has orthonormal columns, so only the middle matrix requires
factorization.  

We compute the thin QR factorization
\begin{equation}
\begin{bmatrix}P_1[u_2,\ldots,u_m]\\I_{m-1}\end{bmatrix}=\widehat Q_{m-1}\widehat R_{m-1},
 \label{eq:small-structured-qr}
\end{equation}
which can be updated by incremental Gram--Schmidt, 
at a cost of $\mathcal O(nm+m^2)$ arithmetic operations at step $m$.
Then the middle factor in
\eqref{eq:Vm-three-factor} satisfies
\[
 \begin{bmatrix}a_1&[u_2,\ldots,u_m]\\0&I_{m-1}\end{bmatrix}
 =
 [g_1,\widehat Q_{m-1}]
 \begin{bmatrix}
 \|a_1\|&a_1^*[u_2,\ldots,u_m]/\|a_1\|_2\\
 0&\widehat R_{m-1}
 \end{bmatrix}.
\]
Substituting this QR factorization into \eqref{eq:Vm-three-factor} yields
\begin{align}
 Q_m
 &=\begin{bmatrix}I_n&0\\0&Q_{m-1}\end{bmatrix}
 [g_1,\widehat Q_{m-1}],
 \label{eq:Qm-update}\\
 R_m
 &=\begin{bmatrix}
 \|a_1\|&\bigl(a_1^*[u_2,\ldots,u_m]/\|a_1\|\bigr)R_{m-1}\\
 0&\widehat R_{m-1}R_{m-1}
 \end{bmatrix}.
 \label{eq:Rm-update}
\end{align}
From the first $n$ rows of \eqref{eq:Qm-update}, we obtain
\begin{equation}
 q_{1,1}=a_1/\|a_1\|, \qquad
 [q_{1,2},\ldots,q_{1,m}]
 =\widehat Q_{m-1}(1\!:\!n,:).
 \label{eq:q-first-row}
\end{equation}
On the other hand, applying \eqref{eq:Rm-update} and $S_m=R_m(:,2\!:\!m)R_{m-1}^{-1}$ gives
\begin{equation}
 S_m=
 \begin{bmatrix}
 a_1^*[u_2,\ldots,u_m]/\|a_1\|\\
 \widehat R_{m-1}
 \end{bmatrix}.
 \label{eq:S-update}
\end{equation}
Thus, we have obtained $q_{1,m}$ and $S_m$ within $\mathcal{O}(nm+m^2)$ arithmetic operations at step $m$.

\paragraph{Updating the recovery vector $u_{m+1}$}
Notice that
\[\begin{bmatrix}
    a_2 &\cdots&a_{m+1}\\
    a_1&\cdots &a_m
\end{bmatrix}=\begin{bmatrix}
    u_2 &\cdots&u_{m+1}\\
    q_{1,1}&\cdots &q_{1,m}
\end{bmatrix}R_m.\]
Let $t_m = R_me_m\in\mathbb C^m$, then
\begin{align*}
    a_{m+1}&=[u_2,\ldots,u_{m+1}]t_m,\\
    a_m&=[q_{1,1},\ldots,q_{1,m}]t_m.
\end{align*}
If we can compute $t_m$, then we can compute $u_{m+1}$ from $[u_2,\ldots,u_{m+1}]t_m = a_{m+1}$, with $\mathcal O(nm)$ arithmetic cost.
By
\[
S_mt_{m-1} = R_m(:, 2\!:\!m)R_{m-1}^{-1}R_{m-1}e_{m-1} = R_m(:, 2\!:\!m)e_{m-1} = t_m,
\]
we can get the recurrence for $t_m$: 
\[
 t_1=\|a_1\|, \qquad t_m=S_mt_{m-1}, \quad m\ge2.
\]

Thus, we have derived incremental recurrences for $q_{1,m}$, $S_m$, and $u_m$ 
that require only $\mathcal{O}(nm+m^2)$ arithmetic operations at step $m$.  
The resulting algorithm is summarized in Algorithm~\ref{alg:toeplitz-infgmres}.

\begin{algorithm}[htbp]
\caption{Toeplitz infinite GMRES (Toeplitz infGMRES)}
\label{alg:toeplitz-infgmres}
\begin{algorithmic}[1]
\REQUIRE Taylor coefficients $A_0,\ldots,A_q$, maximum Arnoldi length
$2\leq p\leq q$, right-hand side $z$, parameter values $s_1,\ldots,s_N$
\ENSURE Approximations $\widetilde x(s_j)$,
$j=1,\ldots,N$
\STATE Compute and store one LU factorization of $A_0$
\STATE Solve $A_0a_1=z$ with the stored factors; set $t_1=\|a_1\|$; apply
$P_1v=v-a_1(a_1^*v)/\|a_1\|_2^2$ implicitly
\FOR{$m=2,\ldots,p$}
  \STATE Solve $A_0a_m=-\sum_{j=1}^{m-1}A_ja_{m-j}$ using the
  stored factors
  \IF{$m=2$}
    \STATE $u_2=a_2/t_1$
  \ELSE
    \STATE $u_m=\bigl(a_m-[u_2,\ldots,u_{m-1}]
    t_{m-1}(1\!:\!m-2)\bigr)/t_{m-1}(m-1)$
  \ENDIF
  \STATE Update
  $\begin{bmatrix}P_1[u_2,\ldots,u_m]\\I_{m-1}\end{bmatrix}
  =\widehat Q_{m-1}\widehat R_{m-1}$ by incremental Gram--Schmidt
  \STATE $S_m=
  \begin{bmatrix}a_1^*[u_2,\ldots,u_m]/\|a_1\|\\
  \widehat R_{m-1}\end{bmatrix}$
  \STATE $t_m=S_mt_{m-1}$
\ENDFOR
\STATE Form $\underline H_p$ from \eqref{eq:H-from-S}
\FOR{$j=1,\ldots,N$}
  \STATE Solve \eqref{eq:small-ls} for $y_{s_j}$
  \STATE $\widetilde x(s_j)=[a_1/\|z\|_2,u_2,\ldots,u_p]y_{s_j}$
\ENDFOR
\end{algorithmic}
\end{algorithm}

Let $F_{\rm LU}(n)$ denote the cost of factoring $A_0$ and $S_{\rm LU}(n)$
the cost of one pair of triangular solves with the stored factors.
If $\operatorname{nnz}(A_j)=\mathcal{O}(n)$, 
the construction of $\{a_m\}_{m=1}^p$ costs $F_{\rm LU}(n)+pS_{\rm LU}(n)+\mathcal{O}(np^2)$
arithmetic operations. The overall arithmetic complexity of Algorithm~\ref{alg:toeplitz-infgmres} is
\[
 F_{\rm LU}(n)+pS_{\rm LU}(n)+\mathcal{O}(np^2+p^3).
\]
Storing $\{a_m\}_{m=1}^p$, $\{u_m\}_{m=2}^p$, $\widehat Q_{p-1}$, $\widehat R_{p-1}$ and $S_p$ requires
$\mathcal{O}(np+p^2)$.
In the complexity bounds below, we omit the common arithmetic cost
$F_{\rm LU}(n)+pS_{\rm LU}(n)$ and the storage of LU-factors.

\subsection{Cancellation in the recovery recurrence}
\label{sec:stability}

Algorithm~\ref{alg:toeplitz-infgmres} can lose stability. The main unstable step is computing $u_m$:
\begin{equation}
\label{eq:u-update}
u_m = \frac{a_m-[u_2,\ldots,u_{m-1}]t_{m-1}(1:m-2)}{t_{m-1}(m-1)}.
\end{equation}
Let
\[
 c_m=[u_2,\ldots,u_{m-1}]t_{m-1}(1\!:\!m-2).
\]
This step subtracts $c_m$ from $a_m$ and divides by
$t_{m-1}(m-1)$.  A standard cancellation indicator is
\[
 \kappa_m^{\rm canc}=
 \frac{\|a_m\|_2+\|c_m\|_2}{\|a_m-c_m\|_2}.
\]
When $\{a_m\}_{m=1}^p$ becomes nearly linearly dependent, $\kappa_m^{\rm canc}$ can be large and
measures the potential amplification caused by the subtraction.  The
unmodified Toeplitz method then propagates the resulting error.

Assume that we have just finished step $m$ of Toeplitz infGMRES. 
Now we wish to determine whether using equation~\eqref{eq:u-update} 
to update $u_{m+1}$ at step $m+1$ will cause severe cancellation. This means that we need 
an indicator to estimate $\kappa_{m+1}^{\rm canc}$, with the known quantities at step $m$.

To estimate this cancellation, we use the computable indicator
\begin{equation}
 \eta_m=
 \begin{cases}
 \dfrac{2\|G_mS_m t_m(1\!:\!m-1)\|_2}
 {|t_m(m)|\,\|u_m\|_2},&|t_m(m)|\,\|u_m\|_2>0,\\[2mm]
 +\infty,&|t_m(m)|\,\|u_m\|_2=0,
 \end{cases}
\label{eq:judge}
\end{equation}
where $G_m\coloneqq [q_{1,1},\ldots,q_{1,m}]$ is the first block row of $Q_m$.

By the definitions of $G_m$ and $S_m$, we have
\[
 G_mS_m=[q_{1,1},\ldots,q_{1,m}]R_m(:,2:m)R_{m-1}^{-1}=[a_2,\ldots,a_m]R_{m-1}^{-1}=[u_2,\ldots,u_m].
\]
Thus, with $c_{m+1}=[u_2,\ldots,u_m]t_m(1\!:\!m-1)$, the numerator of $\eta_m$ is $2\|c_{m+1}\|_2$.
So
\[
 \frac{\eta_m}{\kappa_{m+1}^{\rm canc}}
 =\frac{2\|c_{m+1}\|_2}{\|a_{m+1}\|_2+\|c_{m+1}\|_2}
  \frac{\|u_{m+1}\|_2}{\|u_m\|_2}.
\]
By the reverse triangle inequality, the first ratio on the right differs from
one by at most $1/\kappa_{m+1}^{\rm canc}$.  Hence when $\kappa_{m+1}^{\rm canc}$ is large, 
$\eta_m$ tracks prospective
cancellation up to $\|u_{m+1}\|_2/\|u_m\|_2$ and is of the same order when
the norms of consecutive recovery vectors are comparable. 

\subsection{Dynamic stabilization}
\label{sec:dynamic}
To weaken the cancellation, we recover selected Arnoldi
vectors when the cancellation is detected.
Here we consider how to construct selected Arnoldi vectors.

Because of the Toeplitz structure of $V_m$, there exist vectors $e^{i,m}\in\mathbb{C}^m$ satisfying
\begin{equation}
 q_{i,m}=G_me^{i,m}, \quad 1\le i\le m.
 \label{eq:e-representation}
\end{equation}
Here we can always assume $e^{i,m}$ is padded with zeros to arbitrarily large length.

\begin{lemma}[Block recovery]
\label{lem:block-recovery}
The recurrence of $\{e^{i,m}\}$ is
\begin{align}
 e^{1,m}&=e_m,\label{eq:e-first}\\
e^{i,m}&=\frac{e^{i-1,m-1}-
 \sum_{k=i}^{m-1}S_m(k,m-1)e^{i,k}}
 {S_m(m,m-1)},\quad 2\le i\le m.\label{eq:e-recurrence}
\end{align}
\end{lemma}

\begin{proof}
The relation $e^{1,m}=e_m$ follows directly from the identity $q_{1,m}=G_me_m$. Moreover,
\begin{equation}
 Q_mS_m
 =Q_mR_m(:,2\!:\!m)R_{m-1}^{-1}
 =\begin{bmatrix}[u_2,\ldots,u_m]\\Q_{m-1}\end{bmatrix}.
 \label{eq:Q-S-shift}
\end{equation}
For $2\leq i\leq m$, taking block row $i$ and column $m-1$ of
\eqref{eq:Q-S-shift} and using $q_{i,k}=0$ for $k<i$ gives
\[
 q_{i-1,m-1}=\sum_{k=i}^{m}S_m(k,m-1)q_{i,k}.
\]
Solving for $q_{i,m}$,
using $G_k=G_m(:,1\!:\!k)$, and zero-padding yields the recurrence.
\end{proof}

At a refresh step $m$, we reconstruct $q_{1,m},\ldots,q_{m,m}$ from
\eqref{eq:e-representation} and replace the generator sequence $\{a_k\}_{k=1}^m$ by
\[
 a_i^{\rm new}=q_{m-i+1,m}, \quad i=1,\ldots,m.
\]
Then we continue the Arnoldi process using $\{a_i^{\rm new}\}$ and $t_m^{\rm new}=e_m$.

Suppose that at some step $k\ge m$, we have
\[
 \begin{bmatrix}
 a_k^{\rm new}\\ \vdots\\ a_1^{\rm new}
 \end{bmatrix}
 =Q_kt_k^{\rm new},
\]
while by definition this equation holds at the case $k=m$.

By the recurrence formula of $\{a_i^{\rm new}\}$ and \eqref{eq:recovery-block-action},
\begin{align*}
 a_{k+1}^{\rm new}
 &=-A_0^{-1}\sum_{i=1}^{k}A_i a_{k+1-i}^{\rm new}\\
 &=-A_0^{-1}[A_1,\ldots,A_k]Q_kt_k^{\rm new}
 =[u_2,\ldots,u_{k+1}]t_k^{\rm new}.
\end{align*}
Combining this with \eqref{eq:Q-S-shift} gives
\[
 \begin{bmatrix}
 a_{k+1}^{\rm new}\\ a_k^{\rm new}\\ \vdots\\ a_1^{\rm new}
 \end{bmatrix}
 =
 \begin{bmatrix}
 [u_2,\ldots,u_{k+1}]\\ Q_k
 \end{bmatrix}t_k^{\rm new}
 =Q_{k+1}S_{k+1}t_k^{\rm new}.
\]
Thus, $t_{k+1}^{\rm new}=S_{k+1}t_k^{\rm new}$ preserves the relationship
between the generators and their coordinates, allowing the basic
recurrence to continue unchanged.

With these techniques and the cancellation indicator $\eta_m$ defined in \eqref{eq:judge}, we propose 
Dynamic Toeplitz infGMRES (Dynamic TiGMRES); see Algorithm~\ref{alg:dynamic}. 

\begin{algorithm}[htbp]
\caption{Dynamic Toeplitz infGMRES (Dynamic TiGMRES)}
\label{alg:dynamic}
\begin{algorithmic}[1]
\REQUIRE Inputs of Algorithm~\ref{alg:toeplitz-infgmres}, threshold
$\tau\geq0$
\STATE Initialize as in Algorithm~\ref{alg:toeplitz-infgmres}; set $a_1^{(0)}=a_1$ 
(because $a_1$ will be updated during the iteration); set $G_1=a_1/\|a_1\|$, $e^{1,1}=1$
\FOR{$m=2,\ldots,p$}
  \STATE Compute $a_m$ from the current generators by \eqref{eq:a-recurrence}
  \STATE Recover $u_m$ by \eqref{eq:u-update}, with $u_2=a_2/\|a_1\|$
  \STATE Update the QR factorization \eqref{eq:small-structured-qr} by
  incremental Gram--Schmidt; form $G_m$ and $S_m$ from
  \eqref{eq:q-first-row}--\eqref{eq:S-update}, using the initial $a_1^{(0)}$
  \STATE Set $t_m=S_mt_{m-1}$
  \STATE Compute $e^{i,m}$, $i=1,\ldots,m$, by
  \eqref{eq:e-first}--\eqref{eq:e-recurrence}, and $\eta_m$ by \eqref{eq:judge}
  \IF{$m\ne p$ and $\bigl(m\le\lfloor p/2\rfloor\ \text{or}\ \eta_m\ge\tau\bigr)$}
    \STATE Set $a_i= G_me^{m-i+1,m}$ for $i=1,\ldots,m$, and $t_m= e_m$
  \ENDIF
\ENDFOR
\STATE Form $\underline H_p$ by \eqref{eq:H-from-S}, using the initial $a_1^{(0)}$
\STATE For each $s_j$, solve \eqref{eq:small-ls} for $y_{s_j}$ and return
$\widetilde x(s_j)=[a_1^{(0)}/\|z\|,u_2,\ldots,u_p]y_{s_j}$
\end{algorithmic}
\end{algorithm}

Refreshing at $m=m_1,\ldots,m_r$ Arnoldi steps needs about $2n\sum_{j=1}^r m_j^2$ extra operations, 
because of the effort required for computing the vectors $\{q_{i,m_k}\}, 1\le i \le m_k, 1\le k \le r$.
Algorithm~\ref{alg:dynamic} refreshes during the first half of the iterations and
when $\eta_m\ge\tau$ thereafter, using the indicator and QR update above. The reason for 
refreshing every step during the first half is that precision of the early steps will deeply 
affect the latter steps. And the cost is $2n\sum_{j=2}^{\lfloor p/2\rfloor} j^2\approx \frac{1}{12}np^3$, 
which is only about $\frac{1}{8}$ of the time that refreshing
all steps takes. In practice, a suitably chosen threshold $\tau$ 
can balance stability and efficiency, as we will see in Section~\ref{sec:experiments}.

In the extreme case $\tau=0$, 
Dynamic TiGMRES will refresh at every step, 
leading to its $\frac{2}{3}np^3+\mathcal O(np^2+p^4)$ arithmetic and
$\mathcal O(np+p^3)$ storage complexity. We call this variant fully refreshed TiGMRES. For TOAR infGMRES, two-level orthogonalization 
costs $\mathcal O(nm+m^3)$ at step $m$
\cite{Liu2025}.  Its $m$ dense $n$-by-$m$ basis products add
$2nm^2$ flops per step. The total arithmetic cost over $p$ steps is therefore $\frac{2}{3}np^3+\mathcal O(np^2+p^4)$. 
Thus fully refreshed TiGMRES and TOAR infGMRES have identical arithmetic and storage complexity. The main difference between 
the two algorithms is: At step $m$, fully refreshed TiGMRES uses $G_m$ as the basis to construct the full Arnoldi vectors, while TOAR infGMRES 
uses the orthonormal basis of the column space of $G_m$ to construct them. 
We will see the benefit of directly using $G_m$ as the basis in Section~\ref{sec:implicit-rebasing}, 
where we can give a simple structured representation for $Q_m$ in terms of $G_m$.

Table~\ref{tab:theoretical-complexity} compares basis construction of each algorithm over
$p\leq q$ steps without assuming the separated representation in
Hypothesis~\ref{hyp:separated-functions}. 
Solving the reduced problems and forming $N$ approximations adds
$\mathcal O(N(p^2+np))$ arithmetic for every method.

\begin{table}[H]
\caption{Basis-construction complexities without Hypothesis~\ref{hyp:separated-functions}.
Here $m_1,\ldots,m_r$ denote the refresh steps of Dynamic TiGMRES, 
and we assume $m_i = i+1, \forall i=1,\ldots,\lfloor p/2\rfloor-1$.}
\label{tab:theoretical-complexity}
\centering
\small
\begin{tabular}{lll}
\toprule
Method & Arithmetic complexity & Storage complexity\\
\midrule
TOAR infGMRES
 & $\frac{2}{3}np^3+\mathcal O(np^2+p^4)$
 & $\mathcal O(np+p^3)$\\
Toeplitz infGMRES
 & $\mathcal O(np^2+p^3)$
 & $\mathcal O(np+p^2)$\\
Dynamic TiGMRES
 & $\frac{1}{12}np^3+2n\sum_{i\ge\lfloor p/2\rfloor}^r m_i^2+\mathcal O(np^2+p^4)$
 & $\mathcal O(np+p^3)$\\
fully refreshed TiGMRES
 & $\frac{2}{3}np^3+\mathcal O(np^2+p^4)$
 & $\mathcal O(np+p^3)$\\
\bottomrule
\end{tabular}
\end{table}

\section{Implicitly rebased Toeplitz infGMRES}
\label{sec:implicit-rebasing}
In this section, we present implicitly rebased Toeplitz infGMRES
(IR-TiGMRES) for $A(s)$ with a split form.  

\subsection{A structured representation of the Arnoldi basis}
\label{sec:ir-compact-basis}
Dynamic TiGMRES in Section~\ref{sec:dynamic} uses $G_m$ as the basis to construct 
every Arnoldi vector. This motivates us to represent $Q_m$ in terms of $G_m$.

Let $J_m\in\mathbb C^{m\times m}$ have ones on the first
superdiagonal and zeros elsewhere. Then the Toeplitz structure of $V_m$ gives
\[
(J_m\otimes I_n)V_m=V_mJ_m.
\]
Using the thin QR factorization $V_m = Q_mR_m$, we can derive that 
\[
(J_m\otimes I_n)Q_m=Q_m(R_mJ_mR_m^{-1}).
\]
Let $K_m \coloneqq R_mJ_mR_m^{-1}$, we can prove the following theorem.

\begin{theorem}[Structured representation for $Q_m$]
For every
$1\leq m\leq p$,
\[
  G_m^*G_m+K_m^*K_m=I_m,\qquad
 Q_m=\begin{bmatrix}
  G_m\\G_mK_m\\ \vdots\\G_mK_m^{m-1}
 \end{bmatrix},
\]
\end{theorem}

\begin{proof}
Let $Q_m^{(i)}$ denote the $i$th $n\times m$ block row of $Q_m$,
so that $Q_m^{(1)}=G_m$. The identity
$(J_m\otimes I_n)Q_m=Q_mK_m$ gives
\[
  Q_m^{(i+1)}=Q_m^{(i)}K_m,
  \qquad i=1,\ldots,m-1.
\]
Thus $Q_m^{(i)}=G_mK_m^{i-1}$, proving the stated representation of $Q_m$.

Since $Q_m^*Q_m=I_m$ and $J_m^*J_m=I_m-e_1e_1^*$, we have
\begin{align*}
  K_m^*K_m
  =(Q_mK_m)^*(Q_mK_m)
  &=\big((J_m\otimes I_n)Q_m\big)^*\big((J_m\otimes I_n)Q_m\big)\\
  &=Q_m^*(J_m^*J_m\otimes I_n)Q_m
  =I_m-G_m^*G_m.
\end{align*}
This proves the remaining identity.
\end{proof}

For computation, $G_m$ and $K_m$ can be obtained directly from the
QR factor $\widehat Q_{m-1}$ in \eqref{eq:small-structured-qr}:
\begin{align}
  G_m&=[q_{1,1},\widehat Q_{m-1}(1\!:\!n,:)],\label {eq:G_m}\\
  K_m&=
  \begin{bmatrix}
    0_{(m-1)\times1}&\widehat Q_{m-1}(n+1\!:\!n+m-1,:)\\
    0&0_{1\times(m-1)}
  \end{bmatrix}.
  \label {eq:K_m}
\end{align}
The first formula follows from \eqref{eq:Qm-update}.
Identity \eqref{eq:Rm-update} and direct multiplication yield
\[
  K_m=R_mJ_mR_m^{-1}
  =\begin{bmatrix}R_{m-1}\\0_{1\times(m-1)}\end{bmatrix}
   \begin{bmatrix}
     0_{(m-1)\times1}&R_{m-1}^{-1}\widehat R_{m-1}^{-1}
   \end{bmatrix}
  =\begin{bmatrix}
     0_{(m-1)\times1}&\widehat R_{m-1}^{-1}\\
     0&0_{1\times(m-1)}
   \end{bmatrix}.
\]
The lower block of \eqref{eq:small-structured-qr} gives
$\widehat Q_{m-1}(n+1\!:\!n+m-1,:)\widehat R_{m-1}=I_{m-1}$,
which establishes the stated formula for $K_m$.

\subsection{Computing the recovery vector}
\label{sec:ir-recovery-vector}

\begin{hypothesis}
\label{hyp:separated-functions}
The matrix-valued function $A(s)$ is available in the split form
\begin{equation}
 A(s)=\sum_{\ell=0}^{d-1}
 f_\ell(s)M_\ell,
 \label{eq:separated-functions}
\end{equation}
where $d$ is fixed and small, $f_0(s)\equiv 1$ and each $f_\ell$ is analytic at
$s=0$, and $\operatorname{nnz}(M_\ell)=\mathcal{O}(n)$.
\end{hypothesis}

Write $f_\ell(s)=\sum_{j\geq0}f_{\ell,j}s^j$.  Comparing
\eqref{eq:separated-functions} with the Taylor expansion \eqref{eq:taylor}
gives
\begin{equation}
 A_j=\sum_{\ell=1}^{d-1}f_{\ell,j}M_\ell,
 \quad j\geq1.
 \label{eq:separated-coefficients}
\end{equation}

The split form of $A(s)$ now allows us to express the recovery vector $u_{m+1}$ simply
using $G_m$ and matrix functions of $K_m$.

For each $\ell$, define the analytic divided difference of $f_\ell$ at the origin by
\[
 \phi_\ell(z)=
 \begin{cases}
  \bigl(f_\ell(z)-f_\ell(0)\bigr)/z,&z\ne0,\\
  f_\ell'(0),&z=0.
 \end{cases}
\]

\begin{theorem}[Recovery vector from the structured representation]
\label{thm:implicit-companion-action}
Under Hypothesis~\ref{hyp:separated-functions}, the recovery vector $u_{m+1}$ can be expressed as
\begin{equation}
 u_{m+1}=-A_0^{-1}\sum_{\ell=1}^{d-1}
 M_\ell G_m\phi_\ell(K_m)e_m.
 \label{eq:implicit-companion-action}
\end{equation}
\end{theorem}

\begin{proof}
The $i$th block of the last column of $Q_m$ is
$q_{i,m}=G_mK_m^{i-1}e_m$.  Equation~\eqref{eq:recovery-block-action} gives
\[
 u_{m+1}=-A_0^{-1}\sum_{i=1}^{m}A_iq_{i,m}.
\]
Substitution of \eqref{eq:separated-coefficients}, followed by an interchange
of the two finite sums, gives
\[
 -A_0^{-1}\sum_{\ell=1}^{d-1}M_\ell G_m
 \left(\sum_{i=1}^{m}f_{\ell,i}K_m^{i-1}\right)e_m.
\]
The polynomial in parentheses is the degree-$(m-1)$ Taylor polynomial of
$\phi_\ell$ at $K_m$.  Since $K_m^m=0$, it equals
$\phi_\ell(K_m)$ exactly.
\end{proof}

Formula~\eqref{eq:implicit-companion-action} requires only $G_m$ and the
vectors $\phi_\ell(K_m)e_m$.  In the next section, we will discuss how to evaluate $\phi_\ell(K_m)e_m$ efficiently.

\subsection{Incremental evaluation of matrix-function actions}
\label{sec:scalar-function-updates}

The recovery formula~\eqref{eq:implicit-companion-action} requires the vectors
$\phi_\ell(K_m)e_m$.  We now describe how to update these vectors as $K_m$
gains a column.  Since $K_m^m=0$, evaluating an analytic function at
$K_m$ requires only its Taylor polynomial of degree at most $m-1$.  Moreover,
the nested form of $K_m$ implies that $\phi_\ell(K_{m-1})$ is the leading
principal block of $\phi_\ell(K_m)$ \cite{Higham2010}.
The following application-specific updates attain the $\mathcal O(m^2)$ cost 
per step.

\paragraph{Fixed-degree polynomials}
Let $f$ have a fixed degree $r$. Horner's rule computes $\phi(K_m)e_m$ in $\mathcal{O}(rm^2)=\mathcal{O}(m^2)$ operations.

\paragraph{Rational functions}
Let $f=P/Q$ and assume $Q(0)\ne0$. Since
$P-f(0)Q$ is divisible by $z$, $\phi(K_m)$ can be evaluated as
\[
 \phi(K_m)=Q(K_m)^{-1}
 \left(\frac{P-f(0)Q}{z}\right)(K_m).
\]
The quotient in parentheses is a polynomial.  Because $K_m$ is strictly upper
triangular, $Q(K_m)$ is upper triangular with the nonzero diagonal entry $Q(0)$.
For fixed degrees, updating the last column and solving the triangular system
cost $\mathcal{O}(m^2)$.

\paragraph{Affine square roots}
For $f(z)=\gamma\sqrt{c+\omega z}$ with $c\ne0$, choose the analytic branch by
fixing the scalar value $\sqrt c$, and define
$T_m=(I+(\omega/c)K_m)^{1/2}$.  If
$K_m=[\begin{smallmatrix}K_{m-1}&k\\0&0\end{smallmatrix}]$, then
\[
 T_m=\begin{bmatrix}T_{m-1}&t\\0&1\end{bmatrix},
 \qquad (T_{m-1}+I)t=(\omega/c)k,
\]
which is the scalar-border form of the classical triangular square-root
recurrence \cite{Bjorck1983}.  The divided difference satisfies
\[
 \phi(K_m)e_m=\gamma\frac{\omega}{\sqrt c}
 (T_m+I)^{-1}e_m.
\]
Both $T_{m-1}+I$ and $T_m+I$ are upper triangular with diagonal entries $2$.
The two triangular solves therefore cost $\mathcal{O}(m^2)$.

\paragraph{Exponentials and trigonometric functions}
For $f(z)=\exp(\alpha z)$, we specialize the incremental triangular
evaluation strategy of \cite{Kressner2017} to the coupled
exponential and divided-difference quantities.  Let $r_{13}(z)=P_{13}(z)/Q_{13}(z)$
be the $[13/13]$ Pad\'e approximant to the exponential function $e^z$, normalized by
$P_{13}(0)=Q_{13}(0)=1$.  Using the binary64 threshold
$\theta_{13}=5.371920351148152$ from
\cite[Table~2.3]{Higham2005}, we choose
\[
 \nu=\left\lceil\log_2\max\{1,|\alpha|/\theta_{13}\}\right\rceil,
 \qquad h=\alpha/2^\nu,
\]
so that the scaling is fixed throughout the nested updates.  Define
\[
 D_h(z)=\frac{P_{13}(hz)-Q_{13}(hz)}{z}.
\]
The numerator has zero constant term, so $D_h$ is a polynomial and no
inversion of the singular matrix $K_m$ is involved.  Initialize
\[
 E^{(0)}=Q_{13}(hK_m)^{-1}P_{13}(hK_m),\qquad
 \Phi^{(0)}=Q_{13}(hK_m)^{-1}D_h(K_m),
\]
where $E^{(0)}$ and $\Phi^{(0)}$ can be seen as approximate matrix-function values at $K_m$ 
of functions $e^{hz}$ and $(e^{hz}-1)/z$, respectively.
The scalar identity
$e^{2hz}-1=(e^{hz}-1)(e^{hz}+1)$, with analytic continuation of the divided
differences at zero, yields the recovery recurrences
\[
 E^{(j+1)}=(E^{(j)})^2,\qquad
 \Phi^{(j+1)}=\Phi^{(j)}(E^{(j)}+I),
 \quad j=0,\ldots,\nu-1.
\]
Then $\Phi^{(\nu)}e_m$ approximates the required action $\phi(K_m)e_m$.
Since $K_m$ extends $K_{m-1}$ by one column, the leading blocks of all
intermediate matrices can be retained from the previous step.  Updating only
the new last columns in the polynomial evaluations, solves, and recovery
stages costs $\mathcal O((13+\nu)m^2)$ arithmetic and storage.  Sine and cosine are handled
by fixed linear combinations of the corresponding complex-exponential
actions.

\subsection{Algorithm and complexity}
\label{sec:ir-algorithm}

Algorithm~\ref{alg:ir-toeplitz-infgmres} summarizes the method above. We call this algorithm Implicitly rebased Toeplitz infGMRES (IR-TiGMRES).

Now we analyze the computational complexity of IR-TiGMRES. 
Following the convention in Section~\ref{sec:recurrences}, the bounds below
omit the common arithmetic cost $F_{\rm LU}(n)+pS_{\rm LU}(n)$ and
LU-factor storage. Under Hypothesis~\ref{hyp:separated-functions}, suppose that the
matrix-function actions are computed using the incremental
evaluators of Section~\ref{sec:scalar-function-updates}.
As discussed in Section~\ref{sec:scalar-function-updates}, the matrix-function 
actions require $\mathcal O(dm^2)$
arithmetic operations at step $m$. The products with $G_{m-1}$,
the sparse matrix actions, and the incremental QR update together
require $\mathcal O(dnm+nm+m^2)$ arithmetic operations per step. Summing over
$m=2,\ldots,p$, including initialization, therefore gives
$\mathcal O(d(np^2+p^3))$ arithmetic operations.
The recovery vectors, compact matrices, and nested evaluator data
require $\mathcal O(np+dp^2)$ storage.

\begin{algorithm}[htbp]
\caption{Implicitly rebased Toeplitz infGMRES (IR-TiGMRES)}
\label{alg:ir-toeplitz-infgmres}
\begin{algorithmic}[1]
\REQUIRE Split form \eqref{eq:separated-functions},
maximum Arnoldi length
$2\leq p\leq q$, right-hand side $z$, parameter values $s_1,\ldots,s_N$
and a specified matrix-function evaluator
for every $f_\ell$
\ENSURE Approximations $\widetilde x(s_j)$,
$j=1,\ldots,N$
\STATE Compute and store one LU factorization of $A_0$
\STATE Solve $A_0a_1=z$ with the stored factors; apply
$P_1v=v-a_1(a_1^*v)/\|a_1\|_2^2$ implicitly
\STATE $G_1=a_1/\|a_1\|_2$, $K_1=0$
\FOR{$m=2,\ldots,p$}
  \FOR{$\ell=1,\ldots,d-1$}
    \STATE Evaluate $\widehat v_{\ell,m-1}\approx\phi_\ell(K_{m-1})e_{m-1}$ on the
    specified branch
  \ENDFOR
  \STATE Solve $A_0u_m=-\sum_{\ell=1}^{d-1}
  M_\ell G_{m-1}\widehat v_{\ell,m-1}$ using the stored factors
  \STATE Update
  $\begin{bmatrix}P_1[u_2,\ldots,u_m]\\I_{m-1}\end{bmatrix}
  =\widehat Q_{m-1}\widehat R_{m-1}$ by incremental Gram--Schmidt
  \STATE Form $G_m$, $K_m$ from
  \eqref{eq:G_m}--\eqref{eq:K_m}, and $S_m$ from \eqref{eq:S-update}
\ENDFOR
\STATE Form $\underline H_p$ from \eqref{eq:H-from-S}
\FOR{$j=1,\ldots,N$}
  \STATE Solve \eqref{eq:small-ls} for $y_{s_j}$
  \STATE $\widetilde x(s_j)=[a_1/\|z\|_2,u_2,\ldots,u_p]y_{s_j}$
\ENDFOR
\end{algorithmic}
\end{algorithm}

\begin{remark}[Grouped TOAR infGMRES under the split form]
\label{rem:grouped-toar}
Hypothesis~\ref{hyp:separated-functions} also permits a cheaper companion
action in TOAR infGMRES.  Write $q_{i,m}$ as $q_{i,m}=O_mc^{i,m}$, where
$c^{i,m}\in \mathbb C^m$ and
$O_m\in \mathbb C^{n\times m}$ is the orthonormal basis of $G_m$.  Grouping
the Taylor coefficients by the fixed matrices $M_\ell$ gives
\[
 \sum_{i=1}^m A_iq_{i,m}
 =\sum_{\ell=1}^{d-1} M_\ell O_m
   \left(\sum_{i=1}^m f_{\ell,i}c^{i,m}\right).
\]
Since $d$ is fixed, the coordinate sums cost $\mathcal O(dm^2)$ and the
basis products and sparse matrix actions cost $\mathcal O(dnm)$
per step.  We call this variant Grouped TOAR infGMRES (G-TOAR infGMRES).  The coordinate
orthogonalization of $c^{i,m}$ still needs $\mathcal O(m^3)$ arithmetic and 
$\mathcal O(m^2)$ storage per step.  
Thus, for $p$ steps, G-TOAR infGMRES has arithmetic complexity
$\mathcal O(d(np^2+p^3)+p^4)$ and storage complexity $\mathcal O(np+p^3)$.
Table~\ref{tab:separated-complexity} compares it with IR-TiGMRES using the
same split form. So compared with G-TOAR, IR-TiGMRES shows advantages both in 
arithmetic and storage complexity, especially
when the number of steps $p$ is large.
\end{remark}

\begin{table}[H]
\caption{Basis-construction complexities under
Hypothesis~\ref{hyp:separated-functions}.  The IR-TiGMRES additionally
requires each $f_\ell$ in the discussed area of Section~\ref{sec:scalar-function-updates}.}
\label{tab:separated-complexity}
\centering
\small
\begin{tabular}{lll}
\toprule
Method & Arithmetic complexity & Storage complexity\\
\midrule
G-TOAR infGMRES
 & $\mathcal O(d(np^2+p^3)+p^4)$
 & $\mathcal O(np+p^3)$\\
IR-TiGMRES
 & $\mathcal O(d(np^2+p^3))$
 & $\mathcal O(np+dp^2)$\\
\bottomrule
\end{tabular}
\end{table}

\section{A residual lower bound and parameter scaling}
\label{sec:residual-floor}
The structured representation of $Q_m$ in Section~\ref{sec:ir-compact-basis} also yields a 
residual lower bound for infGMRES.  

We denote the minimized norm
in \eqref{eq:small-ls} by $\mathcal R_m(s)$:
\[
  \mathcal R_m(s)=\min_{y\in\mathbb{C}^m}
  \left\|\left(\Ibar{m}-s\underline H_m\right)y-\|z\|_2 e_1\right\|_2.
\]

Let $K_m$ have the form given in \eqref{eq:K_m}.
Its strict upper triangularity and $G_m^*G_m+K_m^*K_m=I_m$ give
$K_m^m=0$ and $\|K_m\|_2\leq1$. 
\begin{lemma}[Compact formula for the infGMRES residual]
\label{lem:compact-residual}
Let $b_m^*=e_1^*R_m^{-1}$.  For $s\ne0$,
\[
 \frac{\mathcal R_m(s)}{\|z\|_2}
 =\frac{1}{\sqrt{1+\|z\|_2^2\|(sI_m-K_m)^{-*}b_m\|_2^2}}.
\]
\end{lemma}

\begin{proof}
$\mathcal R_m(s)$ is the distance from $\|z\|_2 e_1$
to the column space of $\Ibar{m}-s\underline H_m$. We determine
the direction orthogonal to this column space.

By \eqref{eq:H-from-S}, the nested QR factors, and
$K_m=R_mJ_mR_m^{-1}$, we have
\begin{align*}
 \left(\Ibar{m}-s\underline H_m\right)
 \begin{bmatrix}\|z\|_2&0\\0&R_{m-1}\end{bmatrix}R_m^{-1}
 &=
 \begin{bmatrix}\|z\|_2 b_m^*\\K_m\end{bmatrix} - s\begin{bmatrix}0\\I_{m}\end{bmatrix}\\
 &=
 \begin{bmatrix}\|z\|_2 b_m^*\\K_m-sI_m\end{bmatrix}.
\end{align*}
Because $\begin{bmatrix}\|z\|_2&0\\0&R_{m-1}\end{bmatrix}R_m^{-1}$ is nonsingular, so $\Ibar{m}-s\underline H_m$ and $\begin{bmatrix}\|z\|_2 b_m^*\\K_m-sI_m\end{bmatrix}$ have the same column space.

For $s\ne0$, nilpotency of $K_m$ makes $sI_m-K_m$ nonsingular.
Thus this column space has a one-dimensional orthogonal complement,
which is spanned by
\[
 \begin{bmatrix}
  1\\ \|z\|_2(sI_m-K_m)^{-*}b_m
 \end{bmatrix}.
\]
Then $\mathcal R_m(s)$ is the length of the projection of
$\|z\|_2 e_1$ onto this direction. So
\[
 \frac{\mathcal R_m(s)}{\|z\|_2}
 =
 \frac{1}
 {\sqrt{1+\|z\|_2^2\|(sI_m-K_m)^{-*}b_m\|_2^2}}.
\]
\end{proof}

\begin{theorem}[An infGMRES residual lower bound]
\label{thm:general-residual-floor}
Let $C_m=\|[A_0,\ldots,A_{m-1}]\|_2$.  For every $|s|>1$,
\begin{equation}
 \frac{\mathcal R_m(s)}{\|z\|_2}\geq
 \frac{|s|-1}{\sqrt{(|s|-1)^2+C_m^2}}.
 \label{eq:general-residual-floor}
\end{equation}
Let $C_\infty^2\coloneqq \sum_{j\geq0}\|A_j\|_2^2$, then $C_m\le C_\infty$,
so we can replace $C_m$
by $C_\infty$. Moreover, if $A(s)$ is analytic in a disk centered
at zero with radius greater than one, then $C_\infty < \infty$.
\end{theorem}

\begin{proof}
The generator recurrence \eqref{eq:a-recurrence} gives
$[A_0,\ldots,A_{m-1}]V_m=ze_1^*$, so
\[
 [A_0,\ldots,A_{m-1}]Q_m=zb_m^*.
\]
Taking 2-norm yields $\|z\|_2\|b_m\|_2\leq C_m$.

For $|s|>1$, $\|K_m\|\leq 1$ and $K_m^m=0$ give
\[
 \|(sI_m-K_m)^{-1}\|_2
 =\bigg\|s^{-1}\sum_{j=0}^{m-1}\left(\frac{K_m}{s}\right)^{j}\bigg\|_2
 \leq\sum_{j=0}^{m-1}|s|^{-j-1}\leq(|s|-1)^{-1}.
\]
Now with Lemma~\ref{lem:compact-residual}, we can prove
\eqref{eq:general-residual-floor}.  The inequality $C_m\le C_\infty$ follows from
\[
 C_m^2=\Big\|\sum_{j=0}^{m-1}A_jA_j^*\Big\|_2\leq\sum_{j=0}^{m-1}\|A_j\|_2^2\leq C_\infty^2.
\]
If $A(s)$ is analytic in a disk centered
at zero with radius $R>1$, then choose $1<r<R$ and by using
Cauchy's integral formula, we have
\[\|A_j\|_2=\Big\|\frac{1}{2\pi i}\oint_{|\zeta|=r} \frac{A(\zeta)}{\zeta^{j+1}} \text{d}\zeta\Big\|_2
\leq\max_{|\zeta|=r}\|A(\zeta)\|_2r^{-j},\]
which implies $C_\infty <\infty$.
\end{proof}

\begin{remark}
Notice that when we apply infGMRES to solve $A(s)x(s)=z$ for $|s|>1$, $A(s)$ is always assumed a priori to be analytic
in the disk centered at zero with radius $|s|$, so $C_\infty < \infty$ and the lower bound is strictly positive.
\end{remark}

\paragraph{Effect of scaling}
Appropriate parameter scaling can significantly improve the convergence of
infGMRES \cite{Correnty2024,Liu2025}. The residual lower bound above helps
explain this effect.
For an expansion center $s_0$ with $A(s_0)$ nonsingular, choose a scaling factor $w>0$, set
\[
 \widehat A(\sigma)=A(s_0+w\sigma),\qquad
 \sigma=\frac{s-s_0}{w}.
\]
The scaled Taylor coefficients of $\widehat A(\sigma)$ at $\sigma = 0$ are $\widehat A_j=w^jA^{(j)}(s_0)/j!$.
Under such scaling, 
set
\[
 t=|\sigma|,\qquad
 C=\widehat C_m(w)\coloneqq \|[\widehat A_0,\ldots,\widehat A_{m-1}]\|_2,
\]
then Theorem~\ref{thm:general-residual-floor} shows that 
infGMRES applied to solve $\widehat A(\sigma)x(\sigma) = z$ has the lower bound $\|z\|_2 F(t,C)$, where
\[
 F(t,C)=\frac{t-1}{\sqrt{(t-1)^2+C^2}},\qquad t>1,\quad C>0.
\]

$F(t,C)$ is strictly increasing in $t$ and strictly decreasing in $C$.
For a fixed target $s$, increasing $|w|$ decreases $t$
and makes $\widehat C_m(w)$ non-decreasing, as the latter is because $\sum_{j=0}^{m-1}\widehat A_j\widehat A_j^*$ increases in the
positive semi-definite order. Thus, both changes weaken the residual lower
bound while $t>1$.  This explains why using a large weight can significantly improve the accuracy of infGMRES, especially when the Taylor coefficients
are small.

\section{Numerical experiments}
\label{sec:experiments}

We set $p=q$ in every experiment, so $p$ is both the Taylor degree and the
Arnoldi length. 
All experiments were run in MATLAB R2023b using eight
computational threads on an exclusively allocated Linux node with two
AMD EPYC 9754 processors.

We call Dynamic TiGMRES with $\tau=0$ fully refreshed TiGMRES and use it
as a baseline because it has completely identical arithmetic and storage complexities 
with TOAR infGMRES, as discussed in Section~\ref{sec:dynamic}.  
Dynamic TiGMRES uses $\tau = 100$, because we want to trigger refreshing when the indicator 
suggests a potential loss of approximately two decimal digits.

For the parameterized systems, we use the relative residual
\[
 r_k(s)=\frac{\|A(s)\widetilde x_k(s)-z\|_2}{\|z\|_2}
\]
to measure the convergence.

The reported runtime in each experiment is the median of seven runs.  We first make one untimed
run, then rotate the method order across the timed runs. 

\subsection{Time-delay system}

We consider the time-delay system
\begin{align*}
 \dot x(t)&=A_0x(t)+A_1x(t-1)-bu(t),\\
 y(t)&=C^Tx(t),
\end{align*}
whose transfer function is
\[
 H(s)=C^T(-sI+A_0+A_1e^{-s})^{-1}b.
\]
Let $A(s) = -sI+A_0+A_1e^{-s}$ and we need to solve $A(s)x(s) = b$.
We follow the matrices and parameter choices of the delay example in
\cite{Correnty2024}, 
where the authors use the
expansion center $s_0=0$ and scaling factor
$w=1$ for both $s=0.1$ and $s=0.4$.

In this experiment, we test our four algorithms: fully refreshed TiGMRES, 
Toeplitz infGMRES, Dynamic TiGMRES with $\tau=100$, and IR-TiGMRES.
Figure~\ref{fig:time-delay} shows the residual histories.  At $s=0.1$, all four
methods first reach $10^{-8}$ at iteration~14, with final residuals of order
$10^{-13}$.  At $s=0.4$, fully refreshed TiGMRES, Dynamic TiGMRES, and
IR-TiGMRES first reach $10^{-8}$ at iteration~31, with final residuals lower than $10^{-13}$.
Toeplitz infGMRES instead stops improving at $3.55\times10^{-4}$ because of
the cancellation described in Section~\ref{sec:stability}.  Dynamic TiGMRES and
IR-TiGMRES remain stable.

\begin{figure}[htbp]
\centering
\begin{subfigure}[b]{0.47\textwidth}
  \centering
  \includegraphics[width=\linewidth]{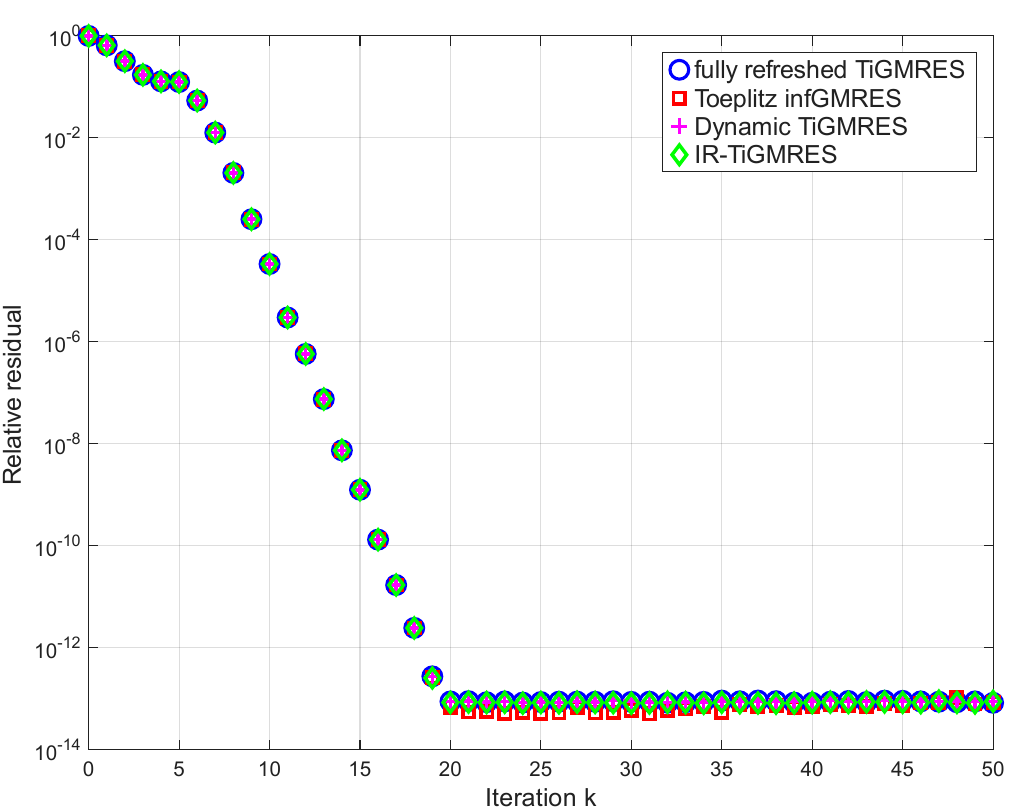}
  \caption{$s=0.1$.}
\end{subfigure}
\hfill
\begin{subfigure}[b]{0.47\textwidth}
  \centering
  \includegraphics[width=\linewidth]{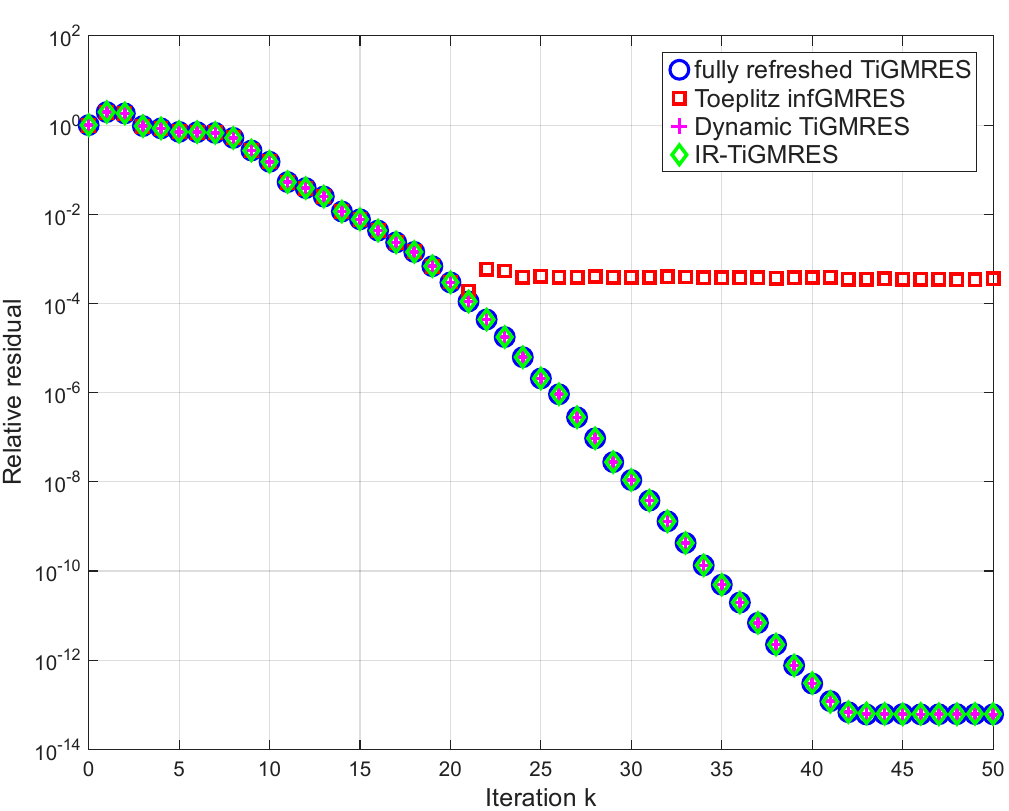}
  \caption{$s=0.4$.}
\end{subfigure}
\caption{Relative residual histories for the time-delay system. 
(a) shows the residual versus iteration results for $s=0.1$ and (b) shows the results for $s=0.4$.}
\label{fig:time-delay}
\end{figure}

\subsection{Parameterized Helmholtz equation}
We consider the boundary-value problem
\[
\begin{aligned}
 \bigl(\nabla^2+f_1(s)(1+sk(x))^2+f_2(s)\beta(x)\bigr)u(x)&=h(x),
 &&x\in\Omega,\\
 u(x)&=0, &&x\in\partial\Omega,
\end{aligned}
\]
where, for $x=(x_1,x_2)$,
\[
k(x)=
\begin{cases}
1+x_1\sin(\alpha\pi x_1),&x_1\in[0,1/2),\\
1+(1-x_1)\sin(\alpha\pi x_1),&x_1\in[1/2,1],
\end{cases}
\]
$h(x)=e^{-\alpha x_1}$, $\beta(x)=\sin(2\pi x_1)$,
$f_1(s)=s$, $f_2(s)=\sin s$, and $\alpha=30$.  Finite element
discretization gives
\[
 A(s)=A_0+A_1s+2A_2s^2+A_3s^3+A_4\sin s,
\]
with sparse coefficient matrices.  We use the same sparse matrices $A_i$ as in \cite{Correnty2024}, and set $s=2.5$,
with expansion center $s_0=0$ and scaling factor $w=5$.

The iteration curves of each algorithm in Figure~\ref{fig:helmholtz}(a) overlap at the plotted
scale.  All four methods first reach $10^{-8}$ at iteration~34 and finish
with relative residuals near $1.47\times10^{-11}$.
Figure~\ref{fig:helmholtz}(b) shows the residual histories versus runtime.
The median end-to-end times are $40.46$s for fully refreshed TiGMRES,
$27.24$s for Toeplitz infGMRES, $30.32$s for Dynamic TiGMRES,
and $19.96$s for IR-TiGMRES.  IR-TiGMRES uses $26.7\%$ less time
than basic Toeplitz, $34.2\%$ less than Dynamic TiGMRES, and
$50.7\%$ less than fully refreshed TiGMRES. 

Although Toeplitz infGMRES and IR-TiGMRES have the same asymptotic $O(np^2+p^3)$
arithmetic complexity, their time performance differs in this test.
Here are the reasons.
For $p=50$, the Taylor convolution in the basic and refreshed Toeplitz
implementations requires $\sum_{k=2}^{p}(k-1)=1225$ sparse matrix-vector
products, whereas the four-term action using the split form in IR-TiGMRES requires only
$4(p-1)=196$. So IR-TiGMRES appears faster than Toeplitz infGMRES. 

Under Hypothesis~\ref{hyp:separated-functions}, a similar grouping operation
can also be easily applied to Toeplitz infGMRES and Dynamic TiGMRES,
but we did not implement it in our experiments.
This is because we regard these two methods as general-purpose
algorithms, rather than methods specifically tailored to a split form with a small $d$.

\begin{figure}[htbp]
\centering
\begin{subfigure}[b]{0.47\textwidth}
  \centering
  \includegraphics[width=\linewidth]{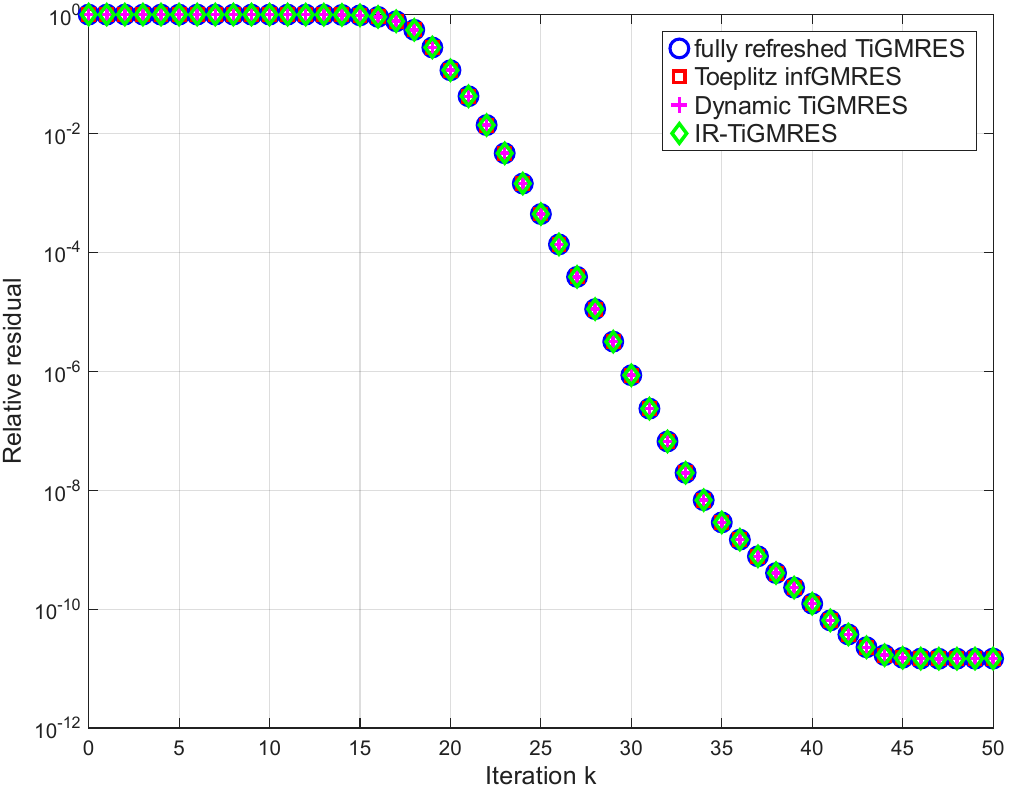}
  \caption{Residual versus iteration.}
\end{subfigure}
\hfill
\begin{subfigure}[b]{0.47\textwidth}
  \centering
  \includegraphics[width=\linewidth]{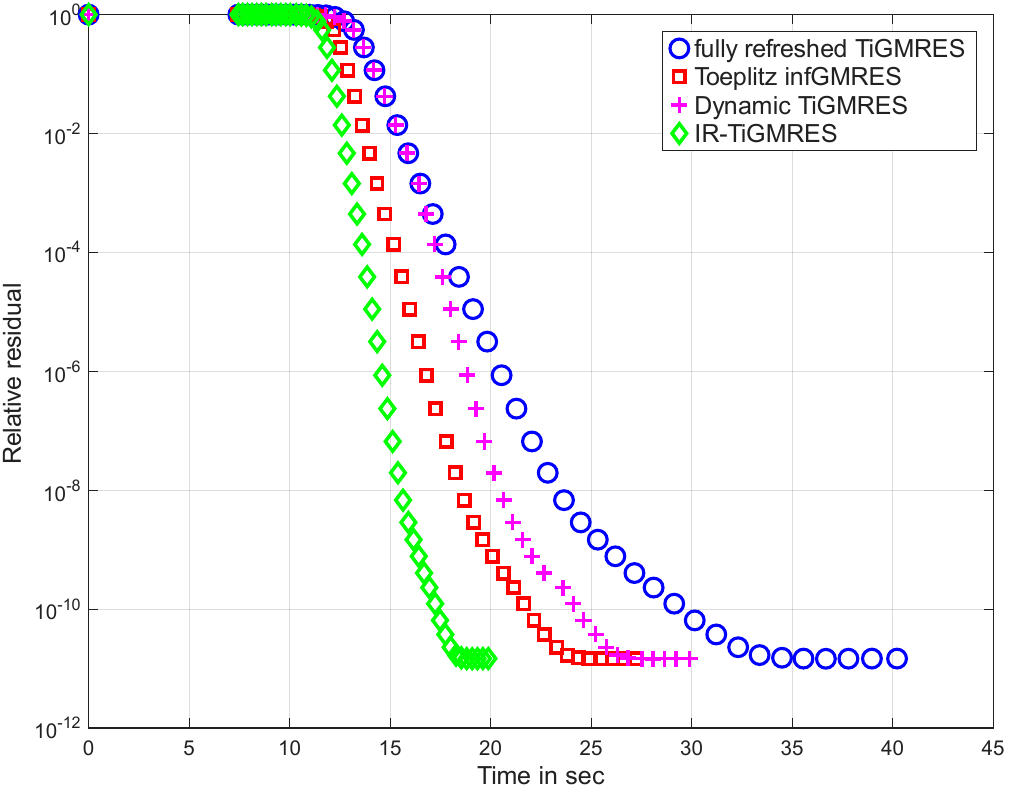}
  \caption{Residual versus runtime.}
\end{subfigure}
\caption{Relative residual histories for the finite element Helmholtz problem with
$s=2.5$.  (a) shows the residual versus iteration results. 
(b) shows the residual versus runtime results.}
\label{fig:helmholtz}
\end{figure}

\subsection{Contour-integral nonlinear eigensolvers}
\label{sec:nep-experiments}

Let $T:\Omega\to\mathbb{C}^{n\times n}$ be analytic, and let the smooth contour
$\Gamma=\partial\Omega$ enclose only simple eigenvalues $\lambda_1,\ldots,\lambda_k$,
and the corresponding eigenvectors $v_1,\ldots,v_k$ of $T$.  

We use the approach in \cite{Liu2025} to compute the eigenpairs.
Beyn's method \cite{Beyn2012} forms
\[
 M_0=\frac{1}{2\pi\mathrm{i}}\oint_\Gamma T(\xi)^{-1}Z\,d\xi,
 \qquad
 M_1=\frac{1}{2\pi\mathrm{i}}\oint_\Gamma \xi T(\xi)^{-1}Z\,d\xi.
\]
For almost any $Z\in \mathbb C^{n\times k}$, if $M_0=V_0\Sigma_0W_0^*$ is the singular 
value decomposition of $M_0$, and construct $\breve M_1 = V_0^*M_1W_0\Sigma_0^{-1}$,
and let $\breve M_1 = S\Lambda S^{-1}$ be the eigenvalue decomposition of $\breve M_1$. Then
$\Lambda = \operatorname{diag}\{\lambda_1,\ldots,\lambda_k\}$ and $V_0S = [v_1,\ldots,v_k]$.  
So the remaining task is to approximate the contour integrals $M_0, M_1$ by using numerical quadrature rules,
which requires solving $T(\xi_j)$ at many quadrature nodes.  These systems
have the same right-hand sides, so infGMRES is a natural choice for solving them efficiently.
\cite{Liu2025}.

\subsubsection{Benchmarks}

We test six problems, which are all chosen from the NLEVP collections\cite{Betcke2013}.  
Figure~\ref{fig:nlevp-contours}
shows their contours, expansion nodes, and computed eigenvalues.

The contours of each problem are the same as in \cite{Liu2025}. But instead of using multiple expansion nodes,
we choose each contour's geometric center
as the only expansion node. This choice can avoid problem-dependent selection of expansion nodes, 
and also allows the Arnoldi process to be run only once.  
However, using fewer expansion nodes may require a larger Taylor degree and Arnoldi length
to improve the accuracy of infGMRES. 

\begin{figure}[!t]
\centering
\begin{subfigure}[t]{0.30\textwidth}
  \centering
  \includegraphics[width=\linewidth]{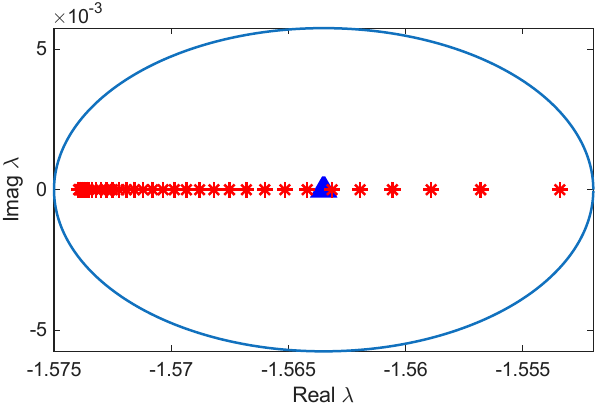}
  \caption{\texttt{spring}.}
\end{subfigure}
\hfill
\begin{subfigure}[t]{0.30\textwidth}
  \centering
  \includegraphics[width=\linewidth]{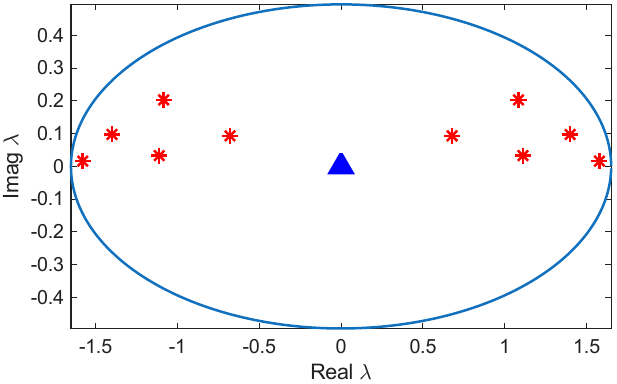}
  \caption{\texttt{acoustic\_wave\_2d}.}
\end{subfigure}
\hfill
\begin{subfigure}[t]{0.30\textwidth}
  \centering
  \includegraphics[width=\linewidth]{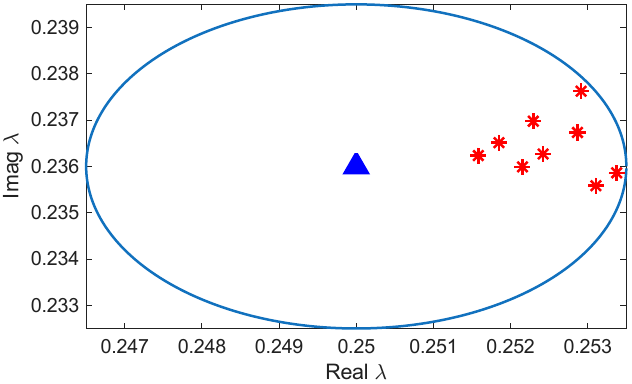}
  \caption{\texttt{butterfly}.}
\end{subfigure}

\par\medskip

\begin{subfigure}[t]{0.30\textwidth}
  \centering
  \includegraphics[width=\linewidth]{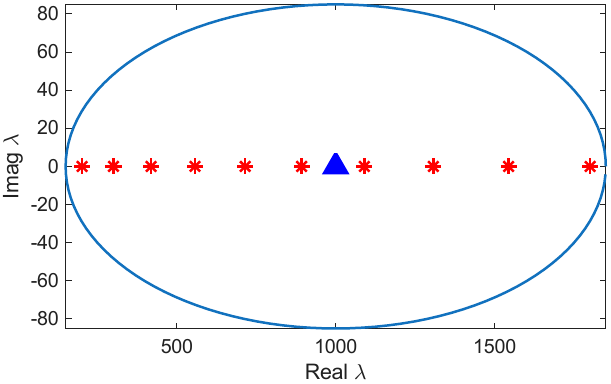}
  \caption{\texttt{loaded\_string}.}
\end{subfigure}
\hfill
\begin{subfigure}[t]{0.30\textwidth}
  \centering
  \includegraphics[width=\linewidth]{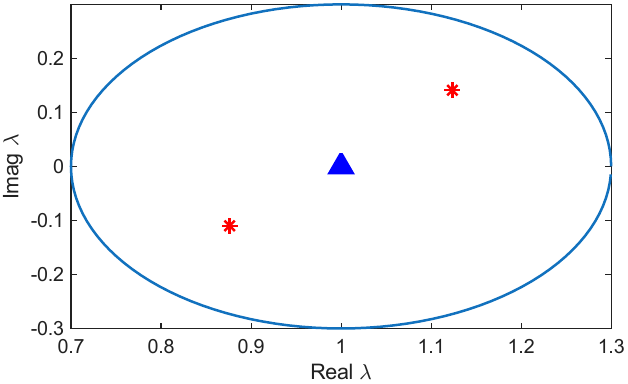}
  \caption{\texttt{railtrack2\_rep}.}
\end{subfigure}
\hfill
\begin{subfigure}[t]{0.30\textwidth}
  \centering
  \includegraphics[width=\linewidth]{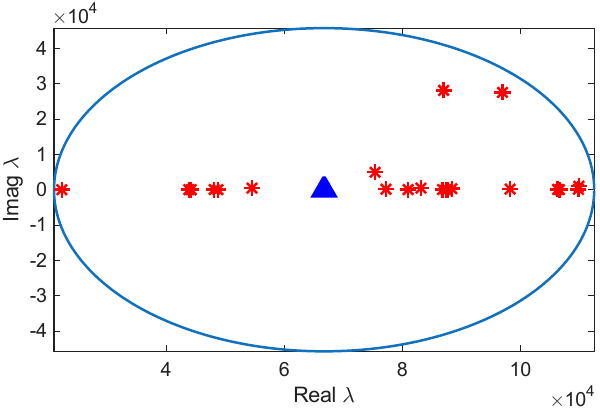}
  \caption{\texttt{gun}.}
\end{subfigure}
\par\medskip
{\small
\setlength{\tabcolsep}{4pt}
\begin{tabular}{@{}cl@{\qquad}cl@{\qquad}cl@{}}
  \raisebox{0.5ex}{\textcolor[rgb]{0,0.447,0.741}{\rule{1.5em}{0.8pt}}}
    & Contour
  & \textcolor[rgb]{0,0.447,0.741}{\large$\blacktriangle$}
    & Expansion node
  & \textcolor{red}{\large$\ast$}
    & Computed eigenvalues
\end{tabular}}
\caption{Contours and expansion nodes for the six NLEVP problems.  The blue
curve is $\Gamma$, the blue triangle is $\lambda_0$, and the red asterisks are
the eigenvalues computed by IR-TiGMRES.  All three methods
use the same contours and
expansion nodes.}
\label{fig:nlevp-contours}
\end{figure}

For a computed eigenpair $(\widehat\lambda,\widehat v)$, we measure the
normwise residual
\begin{equation}
 \rho(\widehat\lambda,\widehat v)=
 \frac{\|T(\widehat\lambda)\widehat v\|_2}
 {\|T(\widehat\lambda)\|_2\|\widehat v\|_2},
 \label{eq:eigenpair-residual}
\end{equation}
where the matrix norm is estimated with MATLAB's \texttt{normest} at tolerance $0.01$.

We test three methods: fully refreshed TiGMRES, Dynamic TiGMRES, and IR-TiGMRES.
For a given problem, all methods use the same contour, expansion node,
scaling factor, Taylor degree, and random seed. 

Table~\ref{tab:nlevp} reports the median end-to-end times.  
Dynamic TiGMRES is faster than fully refreshed TiGMRES in all six problems,
and IR-TiGMRES is fastest in every problem.  Relative to fully refreshed
TiGMRES, their arithmetic mean speedups across the six problems are $1.34$
and $4.39$, respectively.  The smallest speedups occur for
\texttt{railtrack2\_rep}, where the LU factorization dominates the runtime.

Table~\ref{tab:nlevp-scaling-residuals} shows the largest residual for each method, evaluated
from the eigenpairs returned by the last timed run.  The values range from
$5.41\times10^{-14}$ to $1.14\times10^{-10}$.  IR-TiGMRES has the smallest residual in almost every problem. 

\begin{table}[H]
\caption{Problem sizes, scaling factors, runtimes, and speedup factors for the six NLEVP problems.  Here $n$ is the matrix dimension,
$k$ is the requested number of eigenvalues, $N$ is the number of contour nodes,
and $p$ is the Taylor degree.  The scaling factor $w$ is used in
$\widehat T(s)=T(s_0+ws)$.  The times $t_{\rm FR}$, $t_{\rm D}$, and
$t_{\rm IR}$ are medians in seconds for fully refreshed TiGMRES, Dynamic TiGMRES
, and IR-TiGMRES.  The speedup factors are
$S_{\rm D}=t_{\rm FR}/t_{\rm D}$ and
$S_{\rm IR}=t_{\rm FR}/t_{\rm IR}$.}
\label{tab:nlevp}
\centering
\small
\setlength{\tabcolsep}{4pt}
\begin{tabular}{lrrrrrrrrrr}
\toprule
Problem & $n$ & $k$ & $N$ & $p$ & $w$ & $t_{\rm FR}$ & $t_{\rm D}$
& $t_{\rm IR}$ & $S_{\rm D}$ & $S_{\rm IR}$ \\
\midrule
\texttt{spring} & 3000 & 32 & 1000 & 70 & 1.0E-1 & 13.072 & 10.761 & 2.907 & 1.21$\times$ & 4.50$\times$ \\
\texttt{acoustic\_wave\_2d} & 9900 & 10 & 500 & 70 & 1.0E1 & 11.261 & 6.745 & 1.566 & 1.67$\times$ & 7.19$\times$ \\
\texttt{butterfly} & 5041 & 9 & 504 & 50 & 5.0E-1 & 1.975 & 1.610 & 0.695 & 1.23$\times$ & 2.84$\times$ \\
\texttt{loaded\_string} & 20000 & 10 & 128 & 50 & 4.0E3 & 9.692 & 6.968 & 1.800 & 1.39$\times$ & 5.38$\times$ \\
\texttt{railtrack2\_rep} & 35955 & 2 & 128 & 32 & 1.0E1 & 9.651 & 9.399 & 8.798 & 1.03$\times$ & 1.10$\times$ \\
\texttt{gun} & 9956 & 21 & 1000 & 75 & 2.0E5 & 33.208 & 21.716 & 6.263 & 1.53$\times$ & 5.30$\times$ \\
\bottomrule
\end{tabular}
\end{table}

\begin{table}[H]
\caption{Largest eigenpair residuals for the six NLEVP problems.
The columns $\rho_{\max}^{\rm FR}$, $\rho_{\max}^{\rm D}$, and
$\rho_{\max}^{\rm IR}$ give the largest residual
\eqref{eq:eigenpair-residual} for each method.}
\label{tab:nlevp-scaling-residuals}
\centering
\small
\begin{tabular}{lrrr}
\toprule
Problem & $\rho_{\max}^{\rm FR}$ & $\rho_{\max}^{\rm D}$
& $\rho_{\max}^{\rm IR}$ \\
\midrule
\texttt{spring} & $5.41\times10^{-14}$ & $5.45\times10^{-14}$ & $5.54\times10^{-14}$ \\
\texttt{acoustic\_wave\_2d} & $5.05\times10^{-12}$ & $5.23\times10^{-12}$ & $4.89\times10^{-12}$ \\
\texttt{butterfly} & $6.66\times10^{-11}$ & $1.14\times10^{-10}$ & $5.72\times10^{-11}$ \\
\texttt{loaded\_string} & $1.35\times10^{-13}$ & $1.12\times10^{-13}$ & $9.58\times10^{-14}$ \\
\texttt{railtrack2\_rep} & $2.06\times10^{-13}$ & $2.07\times10^{-13}$ & $1.91\times10^{-13}$ \\
\texttt{gun} & $2.76\times10^{-12}$ & $1.28\times10^{-12}$ & $1.54\times10^{-13}$ \\
\bottomrule
\end{tabular}
\end{table}

\section{Conclusions}
\label{sec:conclusions}
In this work, we present Toeplitz infGMRES, which avoids computing full Arnoldi vectors 
occurring in infGMRES. To address the numerical cancellation in this method, 
we develop Dynamic TiGMRES, which constructs full Arnoldi vector whenever the cancellation is detected.
Numerical experiments show that Dynamic TiGMRES is efficient and accurate.

We further develop a structured representation of the Arnoldi basis in infGMRES.
And we apply this structure to the case when $A(s)$ has a split form with a small fixed $d$, 
leading to IR-TiGMRES. Numerical experiments demonstrate computational advantages
of IR-TiGMRES in every problem.

The structured representation also yields a residual lower bound for infGMRES.
We prove that the linearized residual of infGMRES 
cannot converge to zero when $|s|> 1$. This clarifies the role of parameter scaling in 
improving the convergence of infGMRES.

\section*{Acknowledgments}
The authors used OpenAI Codex to refine their original residual lower bound,
which already explained the effect of parameter scaling, into the simpler form
presented in Section~\ref{sec:residual-floor}. Codex also assisted with the
case-by-case discussion of matrix-function actions in Section~\ref{sec:scalar-function-updates}. The authors assume responsibility for all content.

\bibliographystyle{siamplain}
\bibliography{references}
\end{document}